\documentclass[12pt]{article}

\usepackage{mathrsfs}
\usepackage{amsmath}
\usepackage{amsthm}
\usepackage{amssymb}
\usepackage{amsfonts}
\usepackage{epsf}
\usepackage{graphicx}
\usepackage{hyperref}

\usepackage{overpic}

\renewcommand{\l}{\lambda}
\newcommand{\D}{\displaystyle}

\newtheorem{thm}{Theorem}[section]
\newtheorem{lem}{Lemma}[section]

\newtheorem{prop}{Proposition}[section]
\theoremstyle{definition}

\theoremstyle{remark}
\newtheorem{rem}{Remark}[section]

\numberwithin{equation}{section}

\newcommand{\al}{\alpha}
\newcommand{\bt}{\beta}

\begin{document}

\title{Painlev\'e VI and Hankel determinants for the generalized Jacobi weight}
\author{D. Dai \footnotemark[1] \\
\small{\textit{Department of Mathematics, City University of Hong Kong,}} \\
\small{\textit{Tat Chee Avenue, Kowloon, Hong Kong}} \\
\small{dandai@cityu.edu.hk}\\
\vspace{1mm}
and\\
\vspace{1mm} L. Zhang \footnotemark[2]
\\
\small{\textit{Department of Mathematics, Katholieke Universiteit Leuven,}} \\
\small{\textit{Celestijnenlaan 200 B, 3001 Leuven, Belgium}}\\
\small{lun.zhang@wis.kuleuven.be}} \maketitle {
\renewcommand{\thefootnote}{\fnsymbol{footnote}}
\footnotetext[1]{The work of this author was partially supported
by the General Research Fund of Hong Kong \\
(Project No. 9041431)}

\footnotetext[2]{Corresponding author} }

\begin{abstract}
We study the Hankel determinant of the generalized Jacobi weight
$(x-t)^{\gamma}x^\alpha(1-x)^\beta$ for $x\in[0,1]$ with $\alpha,
\beta>0$, $t < 0 $ and $\gamma\in\mathbb{R}$. Based on the ladder
operators for the corresponding monic orthogonal polynomials
$P_n(x)$, it is shown that the logarithmic derivative of Hankel
determinant is characterized by a $\tau$-function for the Painlev\'e
VI system.
\end{abstract}



\section{Introduction and statement of results}\label{introdutcion}

Let $P_n(x)$ be the monic polynomials of degree $n$ in $x$ and
orthogonal with respect to the generalized Jacobi weight $w(x;t)$;
that is
\begin{equation} \label{orthogonality}
\int_0^1 P_m(x) P_n(x) w(x) dx = h_n \delta_{m,n}, \quad h_n>0,
\quad m,n=0,1,2,\cdots,
\end{equation}
where
\begin{equation} \label{pn-formula}
P_n(x) = x^n + \textsf{p}_1(n) x^{n-1} + \cdots
\end{equation}
and
\begin{equation}\label{w-def}
w(x):=w(x;t)=(x-t)^{\gamma}x^\alpha(1-x)^\beta, \qquad  x\in[0,1],
\end{equation}
with $\alpha, \beta>0$, $t < 0$ and $\gamma\in\mathbb{R}$. (In what
follows, we often suppress the $t$-dependence for brevity. We
believe that this will not lead to any confusion.) An immediate
consequence of the orthogonality condition is the three-term
recurrence relation
\begin{equation} \label{recurrence}
x P_n(x) = P_{n+1}(x) + \alpha_n P_n(x) + \beta_n P_{n-1}(x),
\end{equation}
where the ``initial" conditions are taken to be $P_0(x) := 1$ and
$\beta_0 P_{-1}(x) :=0$. Obviously, when $t \to 0^-$, these
polynomials are reduced to the classical Jacobi polynomials up to
some shift and rescaling, whose properties are well-known; see
Szeg\H{o} \cite{s}. In the literature, although people may use
slightly different definitions, the generalized Jacobi polynomials
have been studied from many points of view; for example, see
\cite{martinez,gammel,nevai,Vanlessen,Vertesi}. In particular,
Magnus \cite{magnus} showed that an auxiliary quantity occurring in
his study satisfies the Painlev\'e VI equation for certain
parameters.


In this paper, we are concerned with the Hankel determinant for the
generalized Jacobi weight
\begin{eqnarray}\label{Dn(t)}
D_n(t) & = & \det \left( \int_0^1 x^{j+k} \, w(x;t) dt \right)_{j,k=0}^{n-1} \nonumber \\
& = & \D \frac{1}{n!} \int_0^1 \cdots \int_0^1 \prod_{i<j}(x_i -
x_j)^2 \prod_{k=1}^n w(x_k;t) dx_k
\end{eqnarray}
with $w(x;t)$ given in \eqref{w-def}. Our motivation for this
research mainly arises from the close relation between Hankel
determinants and random matrix theory, which is of interest in
mathematical physics. Indeed, the Hankel determinant defined in
\eqref{Dn(t)} can be viewed as the partition function for the
unitary ensemble with eigenvalue distribution
\begin{equation}
\prod_{i<j}(x_i - x_j)^2 \prod_{k=1}^n
(x_k-t)^{\gamma}x_k^\alpha(1-x_k)^\beta dx_k,
\end{equation}
see the definitive book of Mehta \cite{Mehta} for a discussion of
this topic. The main purpose of this work is to study the properties
of $D_n(t)$ as a function of $t$. More precisely, we are going to
show the logarithmic derivative of Hankel determinant $D_n(t)$ is
characterized by a $\tau$-function for the Painlev\'e VI system. The
appearance of Painlev\'e VI may not be so surprising somehow. As a
matter of fact, it has been already known that, for some special
weight functions, the corresponding Hankel determinants are
connected to the well-known nonlinear ordinary differential
equations -- Painlev\'e equations. In particular, it is first shown
in \cite{haine} that a gap probability in the Jacobi polynomial
ensemble is related to Painlev\'e VI, see also
\cite{BorodinD,ForreW} for further discussion. For the present
general Jacobi case, although it is natural to ``guess" the
existence of such link, the precise form and specific quantity
related to, however, is not clear.

Our approach is based on the ladder operator for orthogonal
polynomials, which has been successfully applied to many other
polynomials before; see \cite{basor-chen,basorCE,cd,ci1,ci2,ci}. The
main result is the following.
\begin{thm} \label{thm-hn}
Let $H_n$ be the logarithmic derivative of the Hankel determinant
with respect to $t$,
\begin{equation} \label{hn-def}
H_n(t):=t(t-1)\frac{d}{dt}\ln D_n(t)
\end{equation}
and denote by
\begin{equation} \label{htil-def}
\widetilde{H}_n:=H_n +d_1 t+d_2
\end{equation}
with
\begin{eqnarray*}
d_1 &= & - n (n + \al + \bt + \gamma)-\frac{(\al + \bt)^2}{4},  \\
d_2 & = & \frac{1}{4} \biggl[ 2n (n + \al + \bt + \gamma) + \bt (\al
+ \bt) - \gamma (\al - \bt) \biggr].
\end{eqnarray*}
Then $\widetilde{H}_n$ satisfies the following Jimbo-Miwa-Okamoto
$\sigma$-form of Painlev\'e VI in \cite{JM,okamoto}
\begin{align}
& \widetilde{H}_n'\biggl(t (t-1) \widetilde{H}_n''\biggr)^2 +
\biggl\{ 2 \widetilde{H}_n' \left( t \widetilde{H}_n' -
\widetilde{H}_n \right) - \widetilde{H}_n'^2 - \nu_1 \nu_2 \nu_3
\nu_4 \biggr\}^2
\nonumber\\
& = \left(\widetilde{H}_n' + \nu_1^2 \right)\left(\widetilde{H}_n' +
\nu_2^2 \right)\left(\widetilde{H}_n' + \nu_3^2
\right)\left(\widetilde{H}_n' + \nu_4^2 \right) \label{sigmaform}
\end{align}
with
\begin{equation*}
\nu_1 = \frac{\al + \bt}{2}, \quad \nu_2 = \frac{\bt - \al}{2},
\quad \nu_3 = \frac{2n + \al + \bt }{2}, \quad \nu_4 = \frac{2n +
\al + \bt +2\gamma}{2}.
\end{equation*}

\end{thm}

\begin{rem}

Due to the symmetric form of \eqref{sigmaform}, the choice of
$\nu_1$, $\nu_2$, $\nu_3$ and $\nu_4$ is not unique.

\end{rem}

\begin{rem}

Although we assume $t<0$ in the definition \eqref{w-def} of the
weight function $w(x)$, Theorem \ref{thm-hn} also holds for $t>1$ if
$(x-t)^\gamma$ is substituted by $(t-x)^\gamma$. One may expect this
theorem is valid for all real $t \neq 0 ,1$ when $(x-t)^\gamma$ is
replaced by $|x-t|^\gamma$, which is similar to what has been
studied by Chen and Feigin \cite{cf}. Unfortunately, we can not
prove it at this moment.

\end{rem}

\begin{rem}

If $\gamma =0$, then $(x-t)^\gamma \equiv 1$ and we can readily
reduce the polynomials $P_n(x)$ in \eqref{orthogonality} to the
classical Jacobi polynomials, which are of course $t$ independent.
As a consequence, the formula \eqref{htil-def} provides a trivial
solution for the associated $\sigma-$form in \eqref{sigmaform}.
Moreover, if $\gamma=1$, the Hankel determinant $D_n(t)$ defined in
\eqref{Dn(t)} is actually a polynomial in $t$ of degree $n$, and
orthogonal with respect to the ``shifted'' Jacobi weight
$t^\al(1-t)^{\beta}$ on $[0,1]$. By the classical theory of Jacobi
polynomials (cf. \cite{s}), $D_n(t)$ satisfies the following
second-order differential equation
\begin{equation}
t(1-t)D_n''-[(2+\al+\beta)t-\al-1]D_n'+n(n+\alpha+\beta+1)D_n=0.
\end{equation}
Hence, if we denote by $u(t):=\frac{d}{dt}\ln D_n(t)$, it is easily
seen that $u(t)$ is a solution of the following Riccati equation
\begin{equation}
t(1-t)u'=t(t-1)u^2+[(2+\al+\beta)t-\al-1]u-n(n+\alpha+\beta+1)
\end{equation}
for $t<0$ in this special case. Moreover, one can verify that
$t(t-1)u(t)+d_1t+d_2$ is a rational solution to the associated
$\sigma-$form in \eqref{sigmaform}.

\end{rem}

From \eqref{recurrence}, it is easily seen that $\beta_n =
h_n/h_{n-1}$. Since
\begin{align} \label{dn-def}
D_n(t) = \prod_{j=0}^{n-1} h_j
\end{align}
(see Eq.(2.1.6) in \cite{ismailbook}), we can express $\bt_n$ in
terms of the Hankel determinant as follows
\begin{equation}
\bt_n = \frac{D_{n-1} D_{n+1}}{D_n^2}.
\end{equation}
Therefore, it is also expected that there exists certain relation
between $\bt_n$ and the Painlev\'e VI equation. Indeed, as a
by-product of our main theorem, we find a first-order differential
equation for $\beta_n$, whose coefficient is closely related to the
Painlev\'e VI equation.
\begin{thm} \label{thm-beta}
The recurrence coefficient $\beta_n$ satisfies a first-order
differential equation as follows
\begin{equation}
t\frac{d}{dt} \, \beta_n  = (2+R_{n-1}-R_n) \, \beta_n,
\end{equation}
where $R_n$ is related to the Painlev\'e VI equation in the
following way. Let
\begin{equation} \label{wn-rn}
W_n(t) := \frac{(t-1)R_n(t)}{2n+\al+\bt+ \gamma+1} + 1.
\end{equation}
Then $W_n(t)$ satisfies the Painlev\'e VI equation
\begin{align}
W_n'' & = \frac{1}{2} \left( \frac{1}{W_n} + \frac{1}{W_n - 1} +
\frac{1}{W_n -t }\right) (W_n')^2 - \left( \frac{1}{t} + \frac{1}{t
- 1} + \frac{1}{W_n -t }\right) W_n' \nonumber \\
& \quad + \frac{W_n (W_n -1) (W_n - t)}{t^2 (t-1)^2} \left( \mu_1 +
\frac{\mu_2 t}{W_n^2} + \frac{\mu_3(t-1)}{(W_n -1)^2} + \frac{\mu_4
t(t-1)}{(W_n -t )^2} \right) \label{w-ode}
\end{align}
with
\begin{equation*}
\mu_1 = \frac{(2n + \al + \bt + \gamma + 1)^2}{2}, \quad \mu_2 =
-\frac{\al^2}{2}, \quad \mu_3 = \frac{\bt^2}{2}, \quad \mu_4 =
\frac{1-\gamma^2}{2}.
\end{equation*}

\end{thm}

The present paper is organized as follows. In Section \ref{sec-pre},
we give a brief introduction to the ladder operator theory and state
three compatibility conditions $S_1$, $S_2$ and $S_2'$. Based on
these supplementary conditions, we introduce some auxiliary
constants in Section \ref{The analysis of the ladder operator}.
Their relations with other quantities such as the coefficient of
orthogonal polynomials, Hankel determinant, etc., are also derived
for further use. We conclude this paper with the proof of Theorem
\ref{thm-hn} and Theorem \ref{thm-beta} in Sections \ref{Proof of
Theorem 1} and \ref{Proof of Theorem 2}, respectively.


\section{Ladder operators and compatibility conditions} \label{sec-pre}
The ladder operators for orthogonal polynomials has been derived by
many authors with a long history, we refer to
\cite{BonanC,BonanLN,Bonan,ci2,sh} and references therein for a
quick guide. Following the general set-up (see \cite{ci2} for
example), we have the lowering and raising ladder operator for our
generalized Jacobi polynomials $P_{n}(z)$:
\begin{align}
\left( \frac{d}{dz} + B_n(z) \right) P_n(z) & = \beta_n A_n(z) P_{n-1}(z), \label{ladder1} \\
\left( \frac{d}{dz} - B_n(z) - \textsf{v}'(z) \right) P_{n-1}(z) & =
- A_{n-1}(z) P_n(z) \label{ladder2}
\end{align}
with $\textsf{v}(z):=-\ln w(z)$ and
\begin{align}
A_n(z) & := \frac{1}{h_n} \int_0^1 \frac{\textsf{v}'(z) -
\textsf{v}'(y)}{z-y} \ [P_n(y)]^2 w(y) dy,
\label{an-def}\\
B_n(z) & := \frac{1}{h_{n-1}} \int_0^1 \frac{\textsf{v}'(z) -
\textsf{v}'(y)}{z-y} \ P_{n-1}(y) P_n(y) w(y) dy. \label{bn-def}
\end{align}
Note that, $A_n(z)$ and $B_n(z)$ are not independent but satisfy the
following supplementary conditions.
\begin{prop}
The functions $A_n(z)$ and $B_n(z)$ defined in (\ref{an-def}) and
(\ref{bn-def}) satisfy the following compatibility conditions,
$$
B_{n+1}(z) + B_n(z)  = (z- \alpha_n) A_n(z) - \textsf{v}'(z),
\eqno(S_1)
$$
$$
1+ (z- \al_n) [B_{n+1}(z) - B_n(z)] = \beta_{n+1} A_{n+1}(z) -
\beta_n A_{n-1}(z).\eqno(S_2)
$$
\end{prop}
\begin{proof}
Using the recurrence relation and the Christoffel-Darboux formulas,
all the formulas (\ref{ladder1}), (\ref{ladder2}), $(S_1)$ and
$(S_2)$ can be derived by direct calculations. We refer to
\cite{BonanC,BonanLN,Bonan,sh} for details. Also, see \cite{ci1,ci2}
for a recent proof.
\end{proof}

From $(S_1)$ and $(S_2)$, we can derive another identity involving
$\sum_{j=0}^{n-1}A_j$ which is very helpful in our subsequent
analysis. We state the result in the following Proposition.

\begin{prop}

$$
B_n^2(z) + \textsf{v}'(z) B_n(z) + \sum_{j=0}^{n-1}A_j(z) =
\beta_{n} A_n(z) A_{n-1}(z). \eqno(S_2')
$$

\end{prop}

\begin{proof}
See the proof of Theorem 2.2 in \cite{cd}.
\end{proof}
The conditions $S_1$, $S_2$ and $S_2'$ are usually called the
compatibility conditions for the ladder operators, which will play
an important role in our future analysis. Although the author
obtained an equivalent form of $S_2'$ in \cite{magnus}, he did not
further study it. We would also like to emphasize that, as in
\cite{cd}, the condition $S_2$ is essential in the present case (see
Remark \ref{use of S2} below), while in \cite{basor-chen} and
\cite{basorCE}, only the conditions $S_1$ and $S_2'$ are sufficient
to derive all the relations.




\section{The analysis of the ladder operators}\label{The analysis of the ladder operator}

\subsection{Some auxiliary constants}

To prove our results stated in Section \ref{introdutcion}, we would
like to introduce some auxiliary constants first. For the weight
function $w(z)$ given in \eqref{w-def}, we know
\begin{equation}
\textsf{v}(z):= - \ln w(z) = -\alpha \ln z-\beta \ln (1-z)-\gamma
\ln (z-t).
\end{equation}
Hence,
\begin{equation} \label{vz'}
\textsf{v}'(z) =  - \frac{\alpha}{z} - \frac{\beta}{z-1} -
\frac{\gamma}{z-t}
\end{equation}
and
\begin{equation} \label{vz'2}
\frac{\textsf{v}'(z) - \textsf{v}'(y)}{z-y} = \frac{\alpha}{z y} +
\frac{\beta}{(z-1)(y-1)} + \frac{\gamma}{(z-t)(y-t)}.
\end{equation}
Since the right-hand side of the above formula is rational in $z$,
it is easily seen that both $A_n(z)$ and $B_n(z)$ are also rational
in $z$ from their definitions in (\ref{an-def}) and (\ref{bn-def}).
More precisely, we have the following lemma.

\begin{lem}
We have
\begin{align}
A_n(z) & = \frac{R^*_n}{z} - \frac{R_n}{z-1} + \frac{R_n-R^*_n}{z-t}, \label{an-new}\\
B_n(z) & = \frac{r^*_n}{z} - \frac{r_n}{z-1} +
\frac{r_n-r^*_n-n}{z-t}, \label{bn-new}
\end{align}
where
\begin{align}
R^*_n & := \frac{\alpha}{h_n} \int_0^1 [P_n(y)]^2 w(y) \frac{dy}{y}, \label{rn1-def}\\
R_n & := \frac{\beta}{h_n} \int_0^1 [P_n(y)]^2 w(y) \frac{dy}{1-y}, \label{r*n1-def} \\
r^*_n & := \frac{\alpha}{h_{n-1}} \int_0^1 P_{n-1}(y) P_n(y) w(y) \frac{dy}{y}, \label{rn2-def} \\
r_n & := \frac{\beta}{h_{n-1}} \int_0^1 P_{n-1}(y) P_n(y) w(y)
\frac{dy}{1-y}. \label{r*n2-def}
\end{align}
\end{lem}

\begin{proof}
Inserting (\ref{vz'2}) into (\ref{an-def}) gives us
\begin{equation} \label{an-1}
\begin{split}
A_n(z) = &~ \frac{1}{h_n} \left[ \frac{\alpha}{z} \int_0^1
[P_n(y)]^2 w(y) \frac{dy}{y}  + \frac{\beta}{z-1} \int_0^1
[P_n(y)]^2 w(y)
\frac{dy}{y-1}  \right. \\
& \left. \quad +\frac{\gamma}{z-t} \int_0^1 [P_n(y)]^2 w(y)
\frac{dy}{y-t} \right].
\end{split}
\end{equation}
Applying integration by parts, we obtain
\begin{equation}
\int_0^1 [P_n(y)]^2 w(y) \textsf{v}'(y) dy = - \int_0^1 [P_n(y)]^2 d
w(y) = \int_0^1 2 P_n'(y) P_n(y) w(y) d y = 0.
\end{equation}
On account of (\ref{vz'}) and the above formula, we have
\begin{equation} \label{int-p}
\gamma \int_0^1 [P_n(y)]^2w(y) \frac{dy}{y-t}  =  \beta\int_0^1
[P_n(y)]^2 w(y) \frac{dy}{1-y} - \alpha \int_0^1 [P_n(y)]^2 w(y)
\frac{dy}{y}.
\end{equation}
A combination of (\ref{an-1}) and (\ref{int-p}) yields
(\ref{an-new}).

In a similar manner, we get (\ref{bn-new}) from (\ref{bn-def}). In
that case, we need to make use of the following equality
\begin{equation}\label{int by parts 2}
\begin{aligned}
\gamma\int_0^1 P_{n-1}(y) P_n(y) w(y) \frac{dy}{y-t}  = &-n h_{n-1}
-\alpha \int_0^1 P_{n-1}(y) P_n(y) w(y)
\frac{dy}{y} \\
& +\beta\int_0^1 P_{n-1}(y) P_n(y) w(y) \frac{dy}{1-y}.
\end{aligned}
\end{equation}
\end{proof}

In view of the compatibility conditions ($S_1$), ($S_2$) and
($S_2'$), one can derive the following relations among the four
auxiliary quantities $R_n, R_n^*, r_n, r_n^*$.

\begin{prop} \label{eqns-r&R}

From ($S_1$), we obtain the following equations
\begin{align}
r^*_{n+1} + r^*_n & = \alpha - \alpha_n R^*_n, \label{r-r1} \\
r_{n+1} + r_n  & = (1- \alpha_n)R_n - \beta, \label{r-r2} \\
tR_n^*-(t-1)R_n &  = 2n+1+\alpha+\beta+\gamma, \label{r-r3}
\end{align}
where the constants $R_n$, $R_n^*$, $r_n$ and $r^*_n$ are defined in
(\ref{rn1-def})--(\ref{r*n2-def}), respectively.

\end{prop}

\begin{proof}
Substituting (\ref{an-new}) and (\ref{bn-new}) into $(S_1)$, we have
\begin{equation}
\begin{aligned}
&B_{n+1}(z) + B_n(z) \\
&= \frac{r^*_{n+1} + r^*_n}{z} - \frac{r_{n+1} + r_n}{z-1}
+\frac{r_{n+1}+r_{n}-r_{n+1}^*-r_n^*-2n-1}{z-t}
\end{aligned}
\end{equation}
and
\begin{align}
&(z-\alpha_n) A_n(z) - \textsf{v}'(z) \nonumber \\
&= (z-\alpha_n) \left[ \frac{R^*_n}{z} - \frac{R_n}{z-1} +
\frac{R_n-R_n^*}{z-t} \right] +
\frac{\alpha}{z} + \frac{\beta}{z-1} + \frac{\gamma}{z-t}  \nonumber \\
&=\frac{\alpha-\alpha_nR_n^*}{z}-\frac{(1- \alpha_n)R_n -
\beta}{z-1}+\frac{(t-\alpha_n)(R_n-R_n^*)+\gamma}{z-t}.
\end{align}
Comparing the coefficients at $O(z^{-1})$, $O((z-1)^{-1})$ and
$O((z-t)^{-1})$ in the above two formulas, we get
\begin{align}
r^*_{n+1} + r^*_n & = \alpha - \alpha_n R^*_n, \\
 r_{n+1} + r_n & = (1-\alpha_n)R_n - \beta,   \\
  r_{n+1}+r_{n}-r_{n+1}^*-r_n^*-2n-1 &=
 (t-\alpha_n)(R_n-R_n^*)+\gamma.
\end{align}
A combination of the above three formulas gives our proposition.
\end{proof}

\begin{prop} \label{prop-beta}
From ($S_2'$), we have the following equations
\begin{equation}
(r^*_n)^2 - \alpha r^*_n  = \beta_n R^*_n R^*_{n-1},\label{r&R}
\end{equation}
\begin{equation}
r_n^2 + \beta r_n  = \beta_n  R_n R_{n-1},\label{r*&R*}
\end{equation}
\begin{equation}
(2n + \bt + \gamma)r_n-(2n + \al +
\gamma)r_n^*+2r_nr_n^*-n(n+\gamma) =\beta_n(R_{n-1}R^*_n  +
R_nR^*_{n-1} ) \label{r&R&R*}
\end{equation}
and
\begin{equation} \label{Rsum}
\begin{aligned}
\sum_{j=0}^{n-1} R_j =&\
(2n+\alpha+\beta+\gamma)(r_n-r_n^*)-n(n+\gamma)
\\
&+\frac{(2n+\alpha+\beta+\gamma)r_n^*+n(n+\beta+\gamma)}{1-t},
\end{aligned}
\end{equation}
where the constants $R_n$, $R_n^*$, $r_n$ and $r^*_n$ are defined in
(\ref{rn1-def})--(\ref{r*n2-def}), respectively.
\end{prop}

\begin{proof}

Again we substitute (\ref{an-new}) and (\ref{bn-new}) into $(S_2')$
to obtain
\begin{equation} \label{s2'lhs}
\begin{split}
B_n^2(z) + \textsf{v}'(z) B_n(z) + \sum_{j=0}^{n-1}A_j(z) &  \\
& \hspace{-130pt} = \frac{(r_n^*)^2-\alpha r_n^*}{z^2} + \frac{r_n^2
+ \beta r_n}{(z-1)^2}
+ \frac{(r_n-r_n^*-n)(r_n-r_n^*-n-\gamma)}{(z-t)^2}  \\
& \hspace{-130pt} \quad + \frac{\alpha r_n-\beta
r_n^*-2r_nr_n^*}{z(z-1)}+ \frac{(r_n-r_n^*-n)(2r_n^*-\alpha)-\gamma
r_n^*}{z(z-t)} \\
&  \hspace{-130pt} \quad - \frac{( r_n-r_n^*-n) (2r_n+\beta) -
\gamma r_n}{ (z-1)(z-t)}  + \sum_{j=0}^{n-1} \left[ \frac{R^*_j}{z}
- \frac{R_j}{z-1} + \frac{R_j-R_j^*}{z-t} \right]
\end{split}
\end{equation}
and
\begin{align}
&\beta_{n} A_n(z) A_{n-1}(z) \nonumber \\
& =  \frac{\beta_n R^*_n R^*_{n-1} }{z^2} +  \frac{\beta_n R_n
R_{n-1}}{(z-1)^2}+ \frac{\beta_n(R_n-R_n^*)
(R_{n-1}-R_{n-1}^*)}{(z-t)^2} \nonumber \\
& \quad -\frac{\beta_n(R_n^*R_{n-1}+R_nR_{n-1}^*)}{z(z-1)}+
\frac{\beta_n(R_n^*R_{n-1}-2R_n^*R_{n-1}^*+R_nR_{n-1}^*)}{z(z-t)}
\nonumber \\
&\quad
+\frac{\beta_n(R_nR_{n-1}^*-2R_nR_{n-1}+R_{n-1}R_n^*)}{(z-1)(z-t)}.
\label{s2'rhs}
\end{align}
Equating the coefficients of the above two formulas at $O(z^{-2})$,
$O((z-1)^{-2})$ and $O((z-t)^{-2})$, we get (\ref{r&R}),
(\ref{r*&R*}) and
\begin{equation}
(r_n-r_n^*-n)(r_n-r_n^*-n-\gamma)=\beta_n(R_n-R_n^*)
(R_{n-1}-R_{n-1}^*),
\end{equation}
respectively. A substitution of (\ref{r&R}), (\ref{r*&R*}) into the
above formula gives us (\ref{r&R&R*}).  At $O((z-1)^{-1})$, note
that
\begin{equation*}
\frac{1}{z} = 1+O(z-1), \qquad \frac{1}{z-t}=\frac{1}{1-t}+O(z-1),
\qquad \textrm{as } z \to 1.
\end{equation*}
It then follows from \eqref{s2'lhs} and \eqref{s2'rhs} that
\begin{equation}
\begin{aligned}
&-2r_nr_n^*-\beta r_n^*+\alpha r_n-\sum_{j=0}^{n-1}R_j -
\frac{(r_n-r_n^*-n)(2r_n+\beta)-\gamma r_n}{1-t} \\
&= -\beta_n \left[ R_n R_{n-1}^*+R_{n-1}R_n^*  +
\frac{-R_nR_{n-1}^*+2R_nR_{n-1}-R_{n-1}R_n^*}{1-t}\right].
\end{aligned}
\end{equation}
Combining (\ref{r&R&R*}) and the above formula, we have
\begin{equation} \label{Rsum0}
\begin{split}
\sum_{j=0}^{n-1} R_j = & \ (2n+\alpha+\beta+\gamma)(r_n-r_n^*)-n(n+\gamma) \\
&
+\frac{2\beta_nR_nR_{n-1}+(2n+\alpha+\beta+\gamma)r_n^*+n(n+\beta+\gamma)-2r_n^2-2\beta
r_n}{1-t}.
\end{split}
\end{equation}
Eliminating $\beta_nR_nR_{n-1}$ in (\ref{Rsum0}) with the aid of
(\ref{r*&R*}), we finally obtain (\ref{Rsum}).
\end{proof}

\begin{rem}\label{use of S2}
From another condition ($S_2$), we get one more equation as follows
\begin{equation} \label{r&r4}
(t-1)(r_{n+1} - r_n) -t (r_{n+1}^* - r_n^*) - t + \alpha_n = 0.
\end{equation}
Rewriting the above formula yields
\begin{equation} \label{r-alphan}
\al_n= t(r_{n+1}^*-r_n^*) -(t-1)(r_{n+1}-r_n) + t.
\end{equation}
Since it follows from \eqref{pn-formula} and \eqref{recurrence} that
\begin{equation} \label{alpha-p1n}
\alpha_n = \textsf{p}_1(n) - \textsf{p}_1(n+1), \qquad n = 0 , 1, 2,
\cdots
\end{equation}
with $\textsf{p}_1(0):=0$, hence, it is easily seen
\begin{equation} \label{alphasum-p1n}
-\sum_{j=0}^{n-1}\al_j=\textsf{p}_1(n).
\end{equation}
Inserting \eqref{r-alphan} into the above formula, we obtain
\begin{equation} \label{p1n-rs}
\textsf{p}_1(n)=(t-1)r_n-tr_n^*-nt,
\end{equation}
where we have made use of the initial conditions
$r_0(t)=r_0^*(t):=0$.
\end{rem}


\subsection{The recurrence coefficients}

Not only the coefficients $A_n(z)$ and $B_n(z)$ of the ladder
operators in \eqref{ladder1} and \eqref{ladder2}, but also the
recurrence coefficients $\al_n$ and $\bt_n$ in \eqref{recurrence}
can be written in terms of the auxiliary quantities $R_n,$ $r_n$ and
$r_n^*$. Here we do not need $R_n^*$ since it is related to $R_n$ in
a simple way; see (\ref{r-r3}).

\begin{lem} \label{lem-ab}
The recurrence coefficients $\al_n$ and $\bt_n$ are expressed in
terms of $R_n,\:r_n$ and $r_n^*$ as follows:
\begin{equation} \label{alpha-rs}
(2n+2+\al+\bt+\gamma)\al_n=2(t-1)r_n-2tr_n^*+(1-t)R_n+(\alpha+\beta+1)t-\beta
\end{equation}
and
\begin{equation} \label{beta-r-r*}
\begin{aligned}
&(2n-1+\al+\bt+\gamma)(2n+1+\al+\bt+\gamma)\bt_n\\
&=[tr_n^*-(t-1)r_n]^2 -(t-1)(2nt+\gamma t+
\bt)r_n+t[(t-1)(2n+\gamma)-\alpha]r_n^*\\
& \ \ \ +n(n+\gamma)(t^2-t).
\end{aligned}
\end{equation}
\end{lem}

\begin{proof}

We use \eqref{r-r1} and \eqref{r-r2} to eliminate $r_{n+1}^*$ and
$r_{n+1}$ in \eqref{r&r4} and get
\begin{equation}
[1 + t R_n^* - (t-1) R_n ] \al_n = t (\al - 2 r_n^*)  - (t-1)(R_n -
\bt - 2 r_n) + t.
\end{equation}
Substituting \eqref{r-r3} into the above formula immediately gives
us \eqref{alpha-rs}.

To derive the formula for $\bt_n$, we need to consider the
identities in Proposition \ref{prop-beta}. Multiplying both sides of
\eqref{r&R} by $t^2$ and eliminating $t^2R_n^*R_{n-1}^*$ with the
aid of \eqref{r-r3}, we have
\begin{equation}\label{btn-1}
t^2((r_n^*)^2-\alpha
r_n^*)=(t-1)\beta_n[(t-1)R_{n-1}R_n+c_{n-1}R_n+c_{n}R_{n-1}]+c_{n-1}c_{n}\beta_n,
\end{equation}
where $c_n:=2n+1+\alpha+\bt+\gamma$. Similarly, we multiply both
sides of \eqref{r&R&R*} by $t$ and us \eqref{r-r3} again to get
\begin{equation*}
t[(2n+\beta+\gamma)r_n-(2n+\alpha+\gamma)r_n^*+2r_nr_n^*-n(n+\gamma)]
= \beta_n[2(t-1)R_{n-1}R_n+c_{n-1}R_n+c_nR_{n-1}].
\end{equation*}
On account of \eqref{r*&R*}, it is readily derived from the above
formula that
\begin{equation}\label{btn-2}
\begin{aligned}
&t[(2n+\beta+\gamma)r_n-(2n+\alpha+\gamma)r_n^*+2r_nr_n^*-n(n+\gamma)] -(t-1)(r_n^2+\bt r_n) \\
 &= \beta_n[(t-1)R_{n-1}R_n+c_{n-1}R_n+c_{n}R_{n-1}].
\end{aligned}
\end{equation}
Substituting \eqref{btn-2} into \eqref{btn-1} yields
(\ref{beta-r-r*}).
\end{proof}

\begin{rem}
For $n=0,$ from \eqref{recurrence} and the definitions of $R_0(t)$,
$r_0(t)$ and $r_0^*(t)$, it follows that
\begin{align*}
\al_0(t) & = \frac{(\al + 1) \  {}_2F_1( \al + 2, -\gamma, \al + \bt
+ 3; \frac{1}{t})}{(\al + \bt
+ 2) \  {}_2F_1( \al + 1, -\gamma, \al + \bt + 2; \frac{1}{t})} , \\
R_0(t)& = \frac{(\al + \bt + 1) \  {}_2F_1( \al + 1, -\gamma, \al +
\bt + 1; \frac{1}{t})}{{}_2F_1( \al +1, -\gamma, \al + \bt + 2;
\frac{1}{t})}, \\
r_0(t) & = r_0^*(t) = 0,
\end{align*}
where ${}_2F_1$ is the hypergeometric function; see
\cite[p.556]{as}. The validity of (\ref{alpha-rs}) at $n=0$ can be
verified directly from the above formulas.

Furthermore, it is easily seen that
\begin{equation} \label{R0-asy}
R_0(t)= \al + \bt + 1 + O(1/t),
\end{equation}
as $t\to -\infty$.
\end{rem}


\subsection{The $t$ dependance}

Recall that our weight function depends on $t$, therefore, all of
the quantities considered in this paper such as the coefficient of
generalized Jacobi polynomials, Hankel determinant, etc., can be
viewed as functions in $t$. In this subsection, we will investigate
their dependance with respect to this parameter. We start with the
study the coefficient $\textsf{p}_1(n)$ in \eqref{pn-formula}.

\begin{lem}

We have
\begin{equation} \label{p1n-diff}
\frac{d}{dt} \textsf{p}_1(n) =  r_n-r_n^*-n.
\end{equation}

\end{lem}

\begin{proof}

From the orthogonal property (\ref{orthogonality}), we know
\begin{equation*}
\int_0^1 P_n(x) P_{n-1}(x) w(x;t) dx = 0.
\end{equation*}
Note that $P_n(x)$ is also dependent on $t$. Taking derivative of
the above formula with respect to $t$ gives us
\begin{equation*}
\int_0^1 \frac{d}{dt} P_n(x) \; P_{n-1}(x) w(x;t) dx + \int_0^1
P_n(x) P_{n-1}(x) \frac{d}{dt}w(x;t) dx = 0.
\end{equation*}
It then follows from (\ref{orthogonality})--(\ref{w-def}) that
\begin{equation*}
h_{n-1} \frac{d}{dt} \textsf{p}_1(n) - \gamma\int_0^1 P_n(x)
P_{n-1}(x) w(x)  \frac{dx}{x-t} = 0.
\end{equation*}
Combining (\ref{rn2-def}), \eqref{r*n2-def} and \eqref{int by parts
2}, we get (\ref{p1n-diff}) immediately.
\end{proof}

By (\ref{p1n-rs}) and the above lemma, it is easily seen that
\begin{equation} \label{p1n-diff-2}
\frac{d}{dt}\textsf{p}_1(n)= r_n-r_n^*-n =r_n  +(t-1)\frac{d}{dt}r_n
-r_n^* -t\frac{d}{dt}r_n^* -n.
\end{equation}
Hence, we obtain the following nice relation between the derivatives
of $r_n$ and $r_n^*$
\begin{equation} \label{r&r-diff}
t\frac{d}{dt}r_n^*= (t-1)\frac{d}{dt}r_n.
\end{equation}

Next, we derive the following property about the Hankel determinant
$D_n(t)$.

\begin{lem} \label{lem-hankel}

We have
\begin{equation} \label{dn-R*}
t \frac{d}{dt} \ln D_n(t) = n(n+\alpha+\bt+\gamma)- \sum_{j=0}^{n-1}
R_j,
\end{equation}
where $R_j$ is defined in (\ref{r*n1-def}).

\end{lem}

\begin{proof}

Differentiating (\ref{orthogonality}) with respect to $t$ yields
\begin{equation}
h_n' = - \gamma\int_0^1 [P_n(x)]^2 w(x) \frac{dx}{x-t}.
\end{equation}
This, together with (\ref{rn1-def}), \eqref{r*n1-def} and
\eqref{int-p} implies
\begin{equation} \label{h'}
h_n' = h_n(R_n^*-R_n).
\end{equation}
Using \eqref{r-r3} to replace $R_n^*$ by $R_n$, we find
\begin{equation}\label{tlnh'}
t \frac{d}{dt} \ln h_n = 2n+1+\al+\bt+\gamma - R_n.
\end{equation}
Then, a combination of \eqref{tlnh'} and (\ref{dn-def}) gives us
\eqref{dn-R*}.
\end{proof}

Finally, we derive differential equations for the recurrence
coefficients $\alpha_n$ and $\beta_n$. They are the non-standard
Toda equations.
\begin{lem} \label{lem-toda}
The recurrence coefficients $\alpha_n$ and $\beta_n$ satisfy the
following differential equations
$$
t\frac{d}{dt} \,\alpha_n  =  \al_n + r_n - r_{n+1}, \eqno(T_1) \\
$$
$$
t\frac{d}{dt} \, \beta_n  = (2+R_{n-1}-R_n) \, \beta_n, \eqno(T_2).
$$
where $R_n$ and $r_n$ are defined in (\ref{r*n1-def}) and
(\ref{r*n2-def}), respectively.
\end{lem}
\begin{proof}
Applying $t\frac{d}{dt}$ to both sides of (\ref{alpha-p1n}), we have
from (\ref{p1n-diff})
\begin{equation*}
t\frac{d}{dt} \,\alpha_n  = t (r_{n} - r_{n}^* - n ) - t (r_{n+1} -
r_{n+1}^* - n - 1).
\end{equation*}
$(T_1)$ then follows from \eqref{r&r4} and the above formula. Using
\eqref{tlnh'}, it is readily seen that
\begin{equation*}
t\frac{d}{dt} \ln \frac{h_n}{h_{n-1}} =2+R_{n-1}-R_n.
\end{equation*}
Note that $\beta_n = h_n / h_{n-1}$, one easily gets $(T_2)$ from
the above formula.
\end{proof}


\section{Proof of Theorem \ref{thm-hn}}\label{Proof of Theorem 1}

Now we are ready to prove our main theorem. The idea is to make use
of Lemma \ref{lem-hankel} to express $r_n^*$ and $r_n$ in terms of
$H_n$ and its derivative with respect to $t$. Then we derive two
independent formulas for $R_n$ in terms of $r_n^*$ and $r_n$ with
the aid of \eqref{beta-r-r*} and ($T_2$). Equating these two
formulas, finally we obtain an equation involving $H_n,$ $H_n'$ and
$H_n''$. We first need the following proposition for $r_n$ and
$r_n^*$.
\begin{prop}
\begin{align}
r_n^*&=-\frac{n(n+\bt+\gamma)+ (t-1) H_n' - H_n}{2n + \al+ \bt + \gamma}, \label{r*-hn} \\
r_n&=\frac{n(n+\al+\gamma)-t H_n'+ H_n}{2n + \al + \bt+ \gamma}.
\label{r-hn}
\end{align}
\begin{proof}
Recalling the definition of $H_n$ in \eqref{hn-def}, we substitute
(\ref{Rsum}) into (\ref{dn-R*}) and get
\begin{equation} \label{hn-r-r*}
\begin{aligned}
H_n&=\biggl[ n(2n+\al+\bt+2\gamma)-(2n+\al+\bt+\gamma)(r_n-r_n^*) \biggr](t-1)\\
&\ \ \ + (2n+\al+\bt+\gamma)r_n^* +n(n+\bt+\gamma),
\end{aligned}
\end{equation}
Taking a derivative of the above formula with respect to $t$, it
then follows from (\ref{r&r-diff}) that
\begin{equation} \label{hn-r*}
(t-1) H_n' - H_n =-n(n+\bt+\gamma)-(2n + \al+\bt+\gamma)r_n^*,
\end{equation}
which gives us (\ref{r*-hn}). Eliminating $r_n^*$ from the above two
formulas, we get (\ref{r-hn}).
\end{proof}
\end{prop}

Next we derive a proposition for $R_n$ as follows.

\begin{prop}
The auxiliary quantity $R_n$ has the following representations
\begin{align}
R_n(t)= &\ \frac{(2n+1+\al+\bt+\gamma)[l(r_n,
r_n^*,t)-t(1-t)r_n'(t)]} {2 k(r_n, r_n^*,t)}, \label{R-r-r*3} \\
\frac{1}{R_n(t)}=& \ \frac{l(r_n, r_n^*,t)+t(1-t)r_n'(t)} {2 (2n+ 1
+\al+\bt+\gamma)(\bt+r_n) r_n} . \label{R-r-r*4}
\end{align}
where
\begin{equation} \label{l-r-r*}
l(r_n, r_n^*,t):=2(1-t)r_n^2+ [(2n-\bt+\gamma)t+2\bt+2tr_n^*]r_n
-(2n+\al+\gamma)tr_n^*-n(n+\gamma)t
\end{equation}
and
\begin{align}\label{k-r-r*}
k(r_n, r_n^*, t):=& \ [tr_n^*-(t-1)r_n]^2 - (t-1)(2nt+\gamma t+
\bt)r_n
\nonumber \\
& + t[(t-1)(2n+\gamma)-\alpha]r_n^*+n(n+\gamma)(t^2-t).
\end{align}
\end{prop}

\begin{proof}

Using (\ref{r*&R*}), we eliminate $R_{n-1}$ in \eqref{btn-2} and
obtain
\begin{equation*}
\begin{aligned}
&(2n+1+\al+\bt+\gamma)\frac{r_n^2+\bt\:r_n}{R_n}+(2n-1+\al+\bt+\gamma)\bt_nR_n
\\
&=t[(2n+\beta+\gamma)r_n-(2n+\alpha+\gamma)r_n^*+2r_nr_n^*-n(n+\gamma)]
-2(t-1)(r_n^2+\bt r_n).
\end{aligned}
\end{equation*}
Replacing $\bt_n$ in the above formula with the aid of
(\ref{beta-r-r*}), we have
\begin{align}
&\frac{2n+1+\al+\bt+\gamma}{R_n}(r_n^2+\bt r_n)+ \frac{k(r_n,
r_n^*,t)}{2n+1+\al+\bt+\gamma}R_n
\nonumber\\
&= 2 (1-t)r_n^2 + [(2n-\bt+\gamma)t+2\bt+ 2r_n^*t]r_n - (2n +
\alpha+\gamma)tr_n^* - n(n+\gamma)t. \label{R-r-r*}
\end{align}

On the other hand, by applying $t\frac{d}{dt}$ to (\ref{beta-r-r*}),
it follows
\begin{align*}
(2&n-1+\al+\bt+\gamma)(2n+1+\al+\bt+\gamma)t\frac{d}{dt}\bt_n \\
=& \ t\Big[ 2t(r_n^*)^2 + 2(1-2t)r_nr_n^*
+2(t-1)r_n^2+\big(2n-\bt+\gamma-(4n+2\gamma)t\big)r_n \\
& \ \ \ -\big(2n+\al+\gamma-(4n+2\gamma)t\big)r_n^*-(t-1)(2nt+\gamma
t+\beta)\frac{d}{dt}r_n \\
& \ \ \ +t\big((t-1)(2n+\gamma)-\al
\big)\frac{d}{dt}r_n^*+n(n+\gamma)(2t-1)\Big ],
\end{align*}
where we have used \eqref{r&r-diff}. A further substitution of
(\ref{r*&R*}), (\ref{beta-r-r*}) and $(T_2)$ into the above formula
yields
\begin{align}
& \frac{(2n-1+\al+\bt+\gamma)(2n+1+\al+\bt+\gamma)}{R_n}(r_n^2+\bt
r_n)
- k(r_n, r_n^*,t)R_n   \nonumber \\
&  = 2 (t-1)r_n^2 - [(2n-\bt+\gamma)t+2\bt+ 2r_n^*t]r_n + (2n +
\alpha+\gamma)tr_n^* \nonumber \\
&\hspace{66pt}+
n(n+\gamma)t+t(1-t)(2n+\al+\bt+\gamma)\frac{d}{dt}r_n.
\label{R-r-r*2}
\end{align}
Formulas \eqref{R-r-r*3} and \eqref{R-r-r*4} now follow from solving
for $R_n$ and $1/R_n$ from (\ref{R-r-r*}) and (\ref{R-r-r*2}).
\end{proof}

Now we are ready to finish the proof of Theorem \ref{thm-hn}.

\smallskip

\noindent\emph{Proof of Theorem \ref{thm-hn}}. Multiplying
(\ref{R-r-r*3}) and (\ref{R-r-r*4}) gives us
\begin{equation} \label{r-r*}
t^2(t-1)^2 [r_n'(t)]^2 = l^2(r_n, r_n^*,t) - 4 k(r_n,
r_n^*,t)(\bt+r_n) r_n,
\end{equation}
where $l(r_n, r_n^*,t)$ and $k(r_n, r_n^*,t)$ are given in
(\ref{l-r-r*}) and (\ref{k-r-r*}), respectively. Recall that $r_n^*$
and $r_n$ can be written in terms of $H_n$ and $H_n'$, see
(\ref{r*-hn}) and (\ref{r-hn}). Therefore, the above formula
actually gives us a non-linear differential equation for $H_n$.
Using (\ref{htil-def}) to replace $H_n$ by $\widetilde{H}_n$, we
finally get (\ref{sigmaform}), which completes the proof of our
theorem. \hfill $\Box$


\section{Proof of Theorem \ref{thm-beta}}\label{Proof of Theorem 2}

We conclude this paper with the proof of Theorem \ref{thm-beta}.

\smallskip

\noindent\emph{Proof of Theorem \ref{thm-beta}}. Firstly, we try to
express $r_n^*$ in terms of $r_n$, $r_n'$, $R_n$ and $R_n'$. To
achieve this, we substitute (\ref{alpha-rs}) into $(T_1)$ and get an
equation involving $r_n$, $r_n'$, $r_n^*$, ${r_n^*}'$, $R_n'$ and
$r_{n+1}$. Then we use (\ref{r-r2}) and \eqref{r&r-diff} to
eliminate $r_{n+1}$ and ${r_n^*}'$. At the end, we arrive at
\begin{align}
r^*_n  = &~ \frac{1}{2} + \frac{1}{2R_n}\biggl( (t-1)R_n' - 2r_n -
(\al + \bt +1) \biggr) + \frac{1}{2t R_n} \biggl( (2n + \al + \bt +
\gamma + 2)\nonumber\\
&\times(2r_n - R_n + \bt) + (R_n + 1) [ 2(t-1)r_n - (t-1)R_n + (\al
+ \bt + 1)t - \bt ] \biggr)
\end{align}

Next, we insert the above formula into (\ref{R-r-r*3}) and
(\ref{R-r-r*4}) and obtain a pair of \emph{linear} equations in
$r_n$ and $r_n'$. Solving this linear system gives us
\begin{align}
r_n  = F(R_n, R_n') \qquad \textrm{and} \qquad r_n'  = G(R_n, R_n'),
\end{align}
where $F(\cdot, \cdot)$ and $G(\cdot, \cdot)$ are functions that can
be explicitly computed. Because the expressions are too complicated,
we have decided not to write them down. Due to the fact that $\D
\frac{d}{dt} F(R_n, R_n') = r_n' = G(R_n, R_n')$, it can be shown
that
\begin{align}
\biggl[&(2n + \al + \bt + \gamma )(2n + \al + \bt + \gamma +1) +
\biggl((2n + \al + \bt + \gamma +1)t - 2(2n + \al + \bt + \gamma )
\nonumber \\
-& 1 \biggr) R_n(t) - (t-1)R_n^2(t) + t(t-1)R_n'(t)\biggr] \Phi(R_n,
R_n', R_n'')=0, \label{Rn-odes}
\end{align}
where $\Phi(\cdot, \cdot, \cdot)$ is a functions that is explicitly
known. Obviously, the above formula yields two differential
equations. One is a first order differential equation, actually a
Riccati equation, whose solution is given by
\begin{equation}
R_n(t) = \frac{(2n+ \al + \bt + \gamma +1) (1+ \l (2n+ \al + \bt +
\gamma +1)(1-t)^{2n + \al + \bt + \gamma})}{1+ \l (2n+ \al + \bt +
\gamma +1)(1-t)^{2n + \al + \bt + \gamma+1}},
\end{equation}
where $\l$ is an integration constant. However, as $t \to -\infty$,
it is easily seen that
\begin{equation}
R_n(t) \to \begin{cases} 0, & \textrm{ if } \l \neq 0 \\
2n+ \al + \bt + \gamma +1,  & \textrm{ if } \l = 0,
\end{cases}
\end{equation}
which violates the result $R_0(t) \sim \al + \bt + 1$ in
\eqref{R0-asy}. So we discard this Riccati equation.

Finally, applying a suitable rescaling and displacement as given in
\eqref{wn-rn}, we obtain the Painlev\'e VI equation \eqref{w-ode}
from $\Phi(\cdot, \cdot, \cdot) = 0 $ in \eqref{Rn-odes}. And this
finishes the proof of our theorem. \hfill $\Box$

\section*{Acknowledgement}
We thank the referees for providing constructive comments to improve
the contents of this paper.


\end{document}